\documentclass[12pt]{amsart}
\usepackage{amssymb, amsmath, mathtools}
\usepackage{url}

\newtheorem{theorem}{Theorem}[section]
\newtheorem{corollary}[theorem]{Corollary}
\newtheorem{prop}[theorem]{Proposition}

\theoremstyle{remark}
\newtheorem{remark}[theorem]{Remark}

\theoremstyle{definition}
\newtheorem{defn}[theorem]{Definition}
\newtheorem{example}[theorem]{Example}

\numberwithin{equation}{section}
\numberwithin{theorem}{section}

\newcommand{\defeq}{\vcentcolon=}

\newcommand{\Psif}{\Psi_{\!f}}

\let\abs=\envert

\newcommand{\gonehat}{\widehat{g_1}}
\newcommand{\gtwohat}{\widehat{g_2}}

\newcommand{\hatT}{\hat{T}}

\newcommand{\fhat}{\hat f}
\newcommand{\ghat}{\hat g}

\newcommand{\hhat}{\hat h}
\newcommand{\psihat}{\hat\psi}
\newcommand{\psiahat}{\widehat{\psi_a}}
\newcommand{\phihat}{\hat \phi}
\newcommand{\Lone}{\Lany{1}}

\newcommand{\alexc}{{\mathcal A}_c(\R)}
\newcommand{\alexx}{{\mathcal A}_c^{2}(\R)}

\newcommand{\balexc}{{\mathcal B}_c(\R)}

\newcommand{\hk}{\mathcal{HK}(\R)}
\newcommand{\Cc}{{\mathcal C}_c(\R)}
\newcommand{\Dc}{{\mathcal D}_c(\R)}
\newcommand{\Rbar}{\overline{\R}}

\newcommand{\intinf}{\int^\infty_{-\infty}}

\newcommand{\R}{{\mathbb R}}
\newcommand{\C}{{\mathbb C}}
\newcommand{\Z}{{\mathbb Z}}

\newcommand{\Sc}{{\mathcal S}(\R)}
\newcommand{\Scp}{{\mathcal S}'(\R)}
\newcommand{\fn}{\!:\!}

\newcommand{\bv}{\mathcal{BV}(\R)}

\providecommand{\abs}[1]{\lvert#1\rvert}
\providecommand{\norm}[1]{\lVert#1\rVert}
\providecommand{\Lany}[1]{L^{#1}(\R)}

\begin{document}


\date{Preprint January 28, 2025.\quad To appear in {\it Poincare Journal of Analysis \& Applications}}

\title[fourier transform]{The Fourier transform with Henstock--Kurzweil 
and continuous primitive integrals}
\author{Erik Talvila}
\address{Department of Mathematics \& Statistics\\
University of the Fraser Valley\\
Abbotsford, BC Canada V2S 7M8}
\email{Erik.Talvila@ufv.ca}

\begin{abstract}
For each $f\!:\!\mathbb{R}\to\mathbb{C}$ that is Henstock--Kurzweil integrable on the real line,
or is a distribution in the completion of the space of Henstock--Kurzweil integrable functions in the
Alexiewicz norm,
it is shown that 
the Fourier transform is the second distributional derivative
of a H\"older continuous function.  The 
space of such Fourier transforms is isometrically isomorphic
to the completion of the Henstock--Kurzweil integrable functions.
There is an exchange theorem, inversion
in norm and convolution results.  Sufficient conditions are given
for an $L^1$ function to have a Fourier transform that is of bounded variation.  
Pointwise inversion of the Fourier transform is proved for functions
in $L^p$ spaces for $1<p<\infty$.
The exchange theorem is used to evaluate an integral that does not appear in published tables.
\end{abstract}

\maketitle

\section{Introduction}\label{sectionintroduction}

In this paper differentiation and integration in spaces of tempered 
distributions are used to define the Fourier transform of $f\fn\R\to\C$ when $f$ is
Henstock--Kurzweil integrable.  All of the results also apply when 
$f$ is in the completion of the space of
Henstock--Kurzweil integrable functions in the Alexiewicz norm, which consists
of the distributions that are the distributional derivative of a function that
is continuous on the extended real line.   

If $f\in L^1(\R)$ then its Fourier
transform is $\fhat(s)=\intinf e^{-ist}f(t)\,dt$ for $s\in\R$.  In this case, $\fhat$
is uniformly continuous on $\R$ and
vanishes at infinity (Riemann--Lebesgue lemma).  See, for example, \cite{folland},
\cite{grafakosclassical}
or \cite{steinweiss} for
basic results in Fourier analysis.  The dual space and set of multipliers of 
$L^1(\R)$ is $L^\infty(\R)$ so
the Fourier transform is well defined and $\abs{\fhat(s)}\leq\sup_{t\in\R}\abs{e^{-ist}}
\intinf\abs{f(t)}\,dt=\norm{f}_1$.  But when the integral $\intinf f(t)\,dt$ is
allowed to converge conditionally then the set of multipliers are the functions
of bounded variation and $\intinf f(t)g(t)\,dt$ will exist for each function $g$
of bounded variation.  But, the complex exponential kernel $t\mapsto e^{-ist}$ is
not of bounded variation (unless $s=0$).  A consequence is poor pointwise behaviour 
of the Fourier
transform.  Indeed, in this case the Fourier transform integral can diverge for each $s\not=0$.
This problem is circumvented by defining the Fourier transform as a derivative
rather than as an integral.

The Fourier transform is defined
as the second distributional derivative of the integral
$\intinf (1-ist-e^{-ist})f(t)\,dt/t^2$.  The kernel is chosen here so that it
of bounded variation and its second derivative with respect to $s$ is $e^{-ist}$.
If $f$ were in $L^1(\R)$ then this integral could be twice differentiated in the
pointwise sense inside the integral to yield the usual definition of the Fourier
transform.

As a function of $s$,
the integral given above is  H\"older continuous but need not be
Lipschitz continuous.  This formulation gives more information about the Fourier
transform than the exchange formula definition for tempered distributions, 
$\langle \fhat,\phi\rangle=\langle f,\phihat\rangle$,
where $\phi$ is a test function.  This leads to an exchange formula valid when
the functions $x\phi(x)$ and
$x\phi'(x)$ are merely of bounded variation, with mild limit conditions at infinity.  
Inversion in
norm with a summability kernel follows as well as analogues of the usual convolution
results.  In this case, the Fourier transform need not have any pointwise values.
It is integrable in the sense of distributional integrals, which convert to
Stieltjes integrals.

The paper is based on the following observation.
Notice that if $f\in L^1(\R)$ then, using the Fubini--Tonelli theorem in 
\eqref{Omega1} and \eqref{vs},
\begin{equation}
\Omega_f(s)  \defeq
\int_{0}^s\int_0^{\sigma_2}\fhat(\sigma_1)\,d\sigma_1\,d\sigma_2
   =  \intinf f(t)\int_{0}^s\int_0^{\sigma_2}e^{-i\sigma_1t}\,d\sigma_1\,d\sigma_2\,dt
=\intinf v_s(t)f(t)\,dt,\label{Omega1}
\end{equation}
where
\begin{equation}
v_s(t)\defeq\int_{0}^s\int_0^{\sigma_2}e^{-i\sigma_1t}\,d\sigma_1\,d\sigma_2
=\int_0^se^{-i\sigma_1t}(s-\sigma_1)\,d\sigma_1=
\frac{1-ist-e^{-ist}}{t^2}=s^2v_1(st).\label{vs}
\end{equation}
Twice differentiating the first integral in \eqref{Omega1}, using the fundamental theorem
of calculus, gives $\fhat(s)$ at each $s\in\R$.  
Twice differentiating the final integral in \eqref{Omega1}, using dominated convergence,
also gives $\fhat(s)$ at each $s\in\R$.  Using the second distributional derivative, 
this formula will be taken as the definition
of the Fourier transform in the completion of the space of Henstock--Kurzweil integrable
functions.

Denote the space of Henstock--Kurzweil integrable functions by $\hk$.
If $f\in\hk$ then the Alexiewicz norm
of $f$ is $\norm{f}=\sup_{x\in\R}\abs{\int_{-\infty}^xf(t)\,dt}=\norm{F}_\infty$, where
$F(x)=\int_{-\infty}^xf(t)\,dt$ is the primitive of $f$ that vanishes at $-\infty$. 
If $f\in L^1(\R)$
then $f\in\hk$ and
$\norm{f}\leq\norm{f}_1$, with equality if and only if $f$ is of one sign almost
everywhere.  The set of functions with an improper Riemann integral (or Cauchy--Lebesgue
integral) are also in $\hk$.  Henstock--Kurzweil integrals are described, for 
example, in \cite{gordon},
\cite{mcleod}, \cite{monteiroslaviktvrdy} and \cite{swartz}.

The extended real line is $\Rbar=[-\infty,\infty]$.  A function $F\fn\Rbar\to\C$ is
continuous on $\Rbar$ if it is continuous on $\R$ and the limits $\lim_{x\to-\infty}F(x)$
and $\lim_{x\to\infty}F(x)$ exist.  This provides a two-point compactification of the
real line.
Denote the space of functions continuous on $\Rbar$ by $C(\Rbar)$.
For example, the function $F(x)=\arctan(x)$ is in $C(\Rbar)$ if we define $F(\pm\infty)
=\pm\pi/2$.  And, no definition at $\pm\infty$ can make the exponential function
or cosine continuous on $\Rbar$.
Note that $C(\Rbar)$ is a Banach space under the uniform norm, $\norm{F}_\infty
=\sup_{x\in\R}\abs{F(x)}=\max_{x\in\Rbar}\abs{F(x)}$.

The Schwartz space, $\Sc$, consists of the functions in $C^\infty(\R)$ such that
for all integers
$m,n\geq 0$,
$x^m\phi^{(n)}(x)\to 0$ as $\abs{x}\to\infty$.  The tempered distributions, denoted
$\Scp$, is the space of continuous linear functionals on $\Sc$.  Results on distributions used here can
be found, for example, in \cite{donoghue}, \cite{folland} or \cite{friedlanderjoshi}.

\begin{defn}\label{defnAcBc}
Define a space of primitives by
$$
\balexc=\{F\fn\Rbar\to\C\mid F\in C(\Rbar) \text{ and }F(-\infty)=0\}.
$$
Define a space of integrable distributions by
$$
\alexc=\{f\in\Scp\mid f=F' \text{ for some } F\in\balexc\}.
$$
\end{defn}

If $f\in\alexc$ then its continuous
primitive integral is $\int_a^bf=F(b)-F(a)$ for all $a,b\in\Rbar$.
The primitive of $f$ is $F\in\balexc$ such that $F'=f$, where the
distributional derivative of $F$ is 
$\langle F',\phi\rangle=-\langle F,\phi'\rangle=-\intinf F(t)\phi'(t)\,dt$ and 
$\phi\in\Sc$.  
The primitive is unique since the only constant function
in $\balexc$ is the zero function. 

This then furnishes a descriptive definition of the continuous primitive integral.
For the Lebesgue integral on the real line the corresponding space of primitives
are the absolutely continuous functions that are also of bounded variation.

The definition $\langle f,\phi\rangle=\intinf f\phi=-\intinf F(t)\phi'(t)\,dt$
shows that $f\in\alexc$ is a tempered distribution, for if $\phi_k\to 0$ as $k\to\infty$
then
$$
\left|\langle f,\phi_k\rangle\right|=\left|\intinf F(t)\phi'_k(t)\,dt\right|
\leq\norm{F}_\infty\intinf\frac{(t^2+1)\abs{\phi'_k(t)}}{t^2+1}\,dt\to 0 \text{ as }
k\to\infty.
$$
Notice that $\alexc$ is a proper subset of $\Scp$ since the Dirac distribution is a
tempered distribution and is the distributional derivative of the Heaviside step
function, which is not continuous.

Note that
$\alexc$ properly contains $L^1(\R)$, functions having an improper Riemann integral,
functions with a Cauchy--Lebesgue integral, and  $\hk$.
As well, $\alexc$ contains distributions that have no pointwise values.
For example, if $F\in\balexc$ is a function of Cantor--Lebesgue type that has pointwise
derivative zero almost everywhere, then the fundamental theorem of calculus formula
$\int_a^bF'=F(b)-F(a)$ holds for the continuous primitive integral, 
but for
the Lebesgue integral, $\int_a^b F'(t)\,dt=0$.
And,
if $F\in\balexc$ is of Weierstrass type with a pointwise derivative existing nowhere
then the Lebesgue integral $\int_a^b F'(t)\,dt$ does not exist, while, as a continuous
primitive integral, we still have $\int_a^bF'=F(b)-F(a)$.
See \cite{talviladenjoy} for background on the continuous primitive integral.
(Also called the distributional Denjoy integral or distributional Henstock--Kurzweil
 integral.)

As in the preceding paragraph, Lebesgue integrals will explicitly show the variable of
integration and differential while distributional integrals will not show these.  An
exception is Theorem~\ref{theoremfubini}.

The function $v_s$ in \eqref{vs} is of bounded variation.  Since the functions of
bounded variation are in the dual space and are multipliers for $\alexc$ 
\cite[\S5]{talviladenjoy}, the function
$\Omega_f$ of \eqref{Omega1} is well-defined.  
It will be shown in Theorem~\ref{theoremOmegaf} that
$\Omega_f$ is H\"older continuous.  The second distributional derivative of
$\Omega_f$ will then be used as the definition of the Fourier transform in $\alexc$.
Of course, every tempered distribution has a Fourier transform via the exchange formula
$\langle \fhat,\phi\rangle=\langle f,\phihat\rangle$ for each $\phi\in\Sc$.  It is
proved in Theorem~\ref{theoremequivalentdefns} that the new definition agrees with this.  
As seen below, rather more specific results can be proved in this subspace of distributions
than with arbitrary tempered distributions.

A kernel such as $\Omega_f$ and a type of second derivative appears in Riemann's work
(\cite[XVI \S8]{zygmundII}).  See also \cite{hahn1925}.
Using pointwise derivatives,
Bochner \cite{bochner} considered various kernels for locally $L^1$ functions;
see Chapter~VI and Note~78, page~288, for further references to early literature.

An attempt was made in \cite{talvilafourierijm} to study Fourier transforms 
with the Henstock--Kurzweil integral.
However, an error \cite{talvilaijmerrata} means most of the results there only apply to
functions with compact support.  The present paper uses completely different techniques
and is the author's way of mending these mistakes.

There are many ways to define Fourier transforms on various subspaces of distributions.
See \cite{chinnaraman} and references.  See \cite{arredondoreyes} and
\cite{arredondomendozreyes} for results
related to the Henstock--Kurzweil integral.  In \cite{mahantaray}, Fourier transforms
are considered with an integral that includes the Henstock--Kurzweil integral.

The outline of the paper is as follows.  In Section~\ref{sectionOmega}. it is shown that
$\Omega_f$ \eqref{Omega1} is H\"older continuous but need not be Lipschitz continuous.  
Asymptotic properties are also 
proved.  In Section~\ref{sectionbanach}. it is shown that the space 
of Fourier transforms is a Banach space isometrically isomorphic to $\alexc$ and $\balexc$
(Definition~\ref{defnAcBc}).  In Section~\ref{sectionFTalexc}. it is shown that the definition
$\fhat=\Omega_f''$ agrees with the distributional definition of Fourier transform
$\langle \fhat,\phi\rangle=\langle f,\phihat\rangle$ ($\phi\in\Sc$).  
In Section~\ref{sectionexchange}.
the exchange formula $\intinf \fhat g=\intinf f\ghat$ is proved for $f\in\alexc$ and
$g$ from a subset of bounded variation.  Inversion in the Alexiewicz norm is 
accomplished via convolution
with a family of summability kernels, Section~\ref{sectioninversion}.  Pointwise inversion
of Fourier transforms for $f\in L^p(\R)$ ($1<p<\infty$) is obtained by twice differentiating
$\Omega_{\fhat}$.  Various examples
are given in Section~\ref{sectionexamples}., including a function in $\hk$ for which
$\intinf e^{-ist}f(t)\,dt$ diverges for all $s\not=0$, a distribution whose
Fourier transform does not have point values, and a Henstock--Kurzweil integrable
function whose Fourier transform is not locally integrable in the sense of distributions.
Interaction of the Fourier transform
and convolution are considered in Section~\ref{sectionconvolution}.
The exchange theorem is used in Example~\ref{exampleintegralcomputation} to compute $\int_0^\infty s^{-\nu}\log(1+y^2/(s-x)^2)\,ds$,
which does not appear in tables.

To keep the paper reasonably self contained we list a few properties of the continuous
primitive integral that are used repeatedly.

\begin{theorem}[Integration by parts, H\"older inequality]\label{theorempartsholder}
Let $f\in\alexc$ and let $g$ be of bounded variation.  Then
\begin{eqnarray}
\intinf fg & = & F(\infty)g(\infty)-\intinf F(t)\,dg(t)\label{parts}\\
\left|\intinf fg\right| & \leq & \norm{f}(\abs{g(\infty)}+{\rm var}(g)).\label{holder}
\end{eqnarray}
\end{theorem}
See \cite[Definition~6]{talviladenjoy} and \cite[Theorem~7]{talviladenjoy}.
A type of Fubini theorem gives a sufficient condition for existence and equality 
of iterated integrals.
\begin{theorem}[Fubini]\label{theoremfubini}
Let $f\in\alexc$.
Let $g\fn\R\times\R\to\R$ be measurable.
Assume (i) for each $y\in\R$ the function
$x\mapsto g(x,y)$  is of bounded variation; (ii) the function $y\mapsto
{\rm var}(g(\cdot,y))$ is in $L^1(\R)$; (iii) there is $M\in L^1(\R)$
such that for each $x\in \R$ we have $|g(x,y)|\leq M(y)$.  Then
$\intinf\intinf f(x)g(x,y)\,dy\,dx=
\intinf\intinf f(x)g(x,y)\,dx\,dy$.
\end{theorem}

A function is regulated if it has a left limit at each point in $(-\infty,\infty]$
and a right limit at each point in $[-\infty,\infty)$.
Both of these results on Stieltjes integrals also apply on subintervals.
See Proposition~6.2 and Proposition~6.3 in \cite{talvilaSlovaca}.
\begin{prop}\label{prop6.2}
Let $A,B,C\fn[-\infty,\infty]\to\R$ such that two functions are of bounded variation and the
other is regulated.  If $B$ and $C$ have no common point of discontinuity then
$$
\intinf A\,d[BC] =\intinf AB\,dC+\intinf AC\,dB.
$$
\end{prop}
\begin{prop}\label{prop6.3}
Let $A$ be regulated and $B$ be absolutely continuous such that $B'\in L^1(\R)$.  Then
$\intinf A\,dB=\intinf A(t)B'(t)\,dt$.
\end{prop}

It will be convenient to use the following notation.
\begin{defn}\label{defnpn}
For $n\in\Z$, let $p_n(x)=x^n$.
\end{defn}

\section{Properties of $\Omega_f$}\label{sectionOmega}
Various properties of $\Omega_f$ are proved that are essential in definition of the Fourier
transform.

\begin{theorem}[Properties of $\Omega_f$]\label{theoremOmegaf}
Let $v_s(t)=(1-ist-e^{-ist})/t^2$ for $s,t\in\R$, with $v_s(0)=s^2/2$.  For $f\in\alexc$ define
$\Omega_f(s)=\intinf v_sf$.  Then\\
(a) $\abs{\Omega_f(s)}\leq c\norm{f}s^2$ where $c=\intinf\abs{v_1'(t)}\,dt={\rm var}(v_1)$.
It follows that $\Omega_f$ is differentiable at $0$ and
$\Omega_f'(0)=0$.  The constant $c$ is the best possible.\\
(b) As $s\to 0$, $\Omega_f(s)\sim F(\infty) s^2/2$, where $F(x)=\int_{-\infty}^xf$.
As well, $\Omega_f(s)=o(s^2)$ as $\abs{s}\to\infty$ and this growth estimate is sharp.\\
(c) There is a constant $k$ such that for each $s\in\R$ and $0<\abs{h}<1/e$ 
$$
\abs{\Omega_f(s+h)-\Omega_f(s)}\leq k\norm{f}(\abs{s}\abs{h}\abs{\log\abs{h}}
+(s^2+1)\abs{h}).
$$
It follows that $\Omega_f$ is continuous on $\R$ and H\"older 
continuous on each compact interval of $\R$, for each exponent $0<\alpha<1$.\\
(d) $\Omega_f$ need not be Lipschitz continuous.\\
(e) Let $1/\rho+1/\sigma=1$ with $1<\rho<\infty$.  Let $u_s(t)=(1-e^{-ist})/(it)$.
Then 
\begin{align*}
&\abs{\Omega_f(s+h)-\Omega_f(s)}\\
&\leq\norm{f}\left[{\rm var}(v_1)h^2
+\norm{u_1'}_\rho\norm{u_1}_\sigma\abs{h}^{2-1/\rho}\abs{s}^{1/\rho}+
\norm{u_1}_\rho\norm{u_1'}_\sigma\abs{h}^{1-1/\rho}\abs{s}^{1+/\rho}\right].
\end{align*}
\end{theorem}

\begin{proof}
(a) A Taylor expansion shows $v_s$ is analytic on $\R$ for each fixed $s\in\R$.
By the H\"older inequality \eqref{holder} it follows that
$\abs{\Omega_f(s)}\leq\norm{f}{\rm var}(v_s)=c\norm{f}s^2$.  (With the form
of the Alexiewicz norm taken here the factor of `2' is dropped.)

To show that the constant $c$ is the best possible, write $v_s'(t)=\abs{v_s'(t)}
e^{i\theta_s(t)}$ where $-\pi<\theta_s\leq\pi$. Let $a>0$ and define
$$
F_{s,a}(t)=\left\{\begin{array}{cl}
-e^{-i\theta_s(t)}, & \abs{t}\leq a\\
-e^{-i\theta_s(t)}(1+a-\abs{t}), & a\leq\abs{t}\leq a+1\\
0, & \abs{t}\geq a+1.
\end{array}
\right.
$$
Then $F_{s,a}\in\balexc$.
Let $f_{s,a}=F_{s,a}'$.  By dominated convergence, 
$$
\Omega_{f_{s,a}}(s)=\intinf\abs{v_s'(t)}\,dt-\int_{\abs{t}>a}\abs{v_s'(t)}\,dt
+\int_{a<\abs{t}<a+1}\abs{v_s'(t)}(1+a-\abs{t})\,dt\to cs^2 \text{ as } a\to\infty.
$$

(b) For $s\not=0$, integrate by parts \eqref{parts} and use
dominated convergence,
\begin{eqnarray*}
\Omega_f(s) & = & -s^2{\rm sgn}(s)\intinf v_1'(t)F(t/s)\,dt\\
 & = &  -s^2{\rm sgn}(s)\left(\int_{-\infty}^0v_1'(t)F(t/s)\,dt
+\int_0^\infty v_1'(t)F(t/s)\,dt\right)
  \,\,\substack{\sim\\s\to 0}\,\,\,\,  s^2F(\infty)/2.
\end{eqnarray*}
As well,
$$
\lim_{\abs{s}\to\infty}\frac{\abs{\Omega_f(s)}}{s^2}=\left|\intinf v_1'(t)
\lim_{\abs{s}\to\infty}F(t/s)\,dt\right|=0.
$$

To show this is sharp, let $\psi\fn[0,\infty)\to(0,\infty)$ satisfy $\psi(s)=o(s^2)$ 
as $s\to\infty$.  Define the family of linear operators $L_s\fn\alexc\to\C$ by
$L_s[f]=\Omega_f(s)/\psi(s)$.  The estimate $\abs{L_s[f]}\leq c\norm{f}s^2/\psi(s)$ shows
that for each $s>0$, $L_s$ is a bounded linear operator.  For $s>0$ let $F_s(t)=\overline{v_1'(st)}$
and $f_s=F_s'$.  Then $f_s\in\alexc$ and
$$
\frac{\abs{L_s[f_s]}}{\norm{f_s}}\geq\frac{s^2\norm{v_1'}_2^2}{2\psi(s)\norm{v_1'}_\infty}\to\infty \text{ as } s\to\infty.
$$
Hence, the family of linear operators $L_s$ for $s>0$ is not uniformly bounded.
By the uniform boundedness principle it is not pointwise bounded.  There then exists
$f\in\alexc$ such that $\Omega_f(s)/\psi(s)\nrightarrow 0$ as $s\to\infty$. The order relation
$\Omega_f(s)=o(s^2)$ ($\abs{s}\to\infty$) is then sharp.  Since the space $\hk$
is barrelled \cite{swartz}, $f$ can be chosen here.

(c) From the definition of $\Omega_f$ and the H\"older inequality \eqref{holder},
\begin{equation}
\abs{\Omega_f(s+h)-\Omega_f(s)}
\leq\norm{f}\intinf\abs{v_{s+h}'(t)-v_s'(t)}\,dt.\label{vsholder}
\end{equation}
For $s,h>0$,
use \eqref{vs} to compute
$$
\abs{v_{s+h}'(t)-v_s'(t)}\leq \int_{s}^{s+h}\int_0^{\sigma_2}\sigma_1\,d\sigma_1\,d\sigma_2
=s^2h/2+sh^2/2+h^3/6.
$$
Then, for $h,s\in\R$,
\begin{equation}
\int_{-1}^1\abs{v_{s+h}'(t)-v_s'(t)}\,dt\leq s^2\abs{h}+\abs{s}h^2+\abs{h}^3/3.\label{vs0}
\end{equation}

From \eqref{vs}, with $t\not=0$,
\begin{equation}
v_{s+h}'(t)-v_s'(t) =\frac{ih(1+e^{-i(s+h)t})}{t^2}-\frac{ise^{-ist}(1-e^{-iht})}{t^2}
-\frac{2e^{-ist}(1-e^{-iht})}{t^3}.\label{vs2}
\end{equation}
The first part of the right side of this expression gives
\begin{equation}
\int_{\abs{t}>1}\left|\frac{ih(1+e^{-i(s+h)t})}{t^2}\right|\,dt\leq 2\abs{h}\int_{\abs{t}>1}
\frac{dt}{t^2}=4\abs{h}.\label{vs3}
\end{equation}
For the second term on the right of \eqref{vs2}, write
\begin{align}
&\int_{\abs{t}>1}\left|\frac{ise^{-ist}(1-e^{-iht})}{t^2}\right|\,dt
  =  4\abs{s}\int_{1}^\infty\frac{\abs{\sin(ht/2)}}{t^2}\,dt
 =2\abs{s}\abs{h}\int_{\abs{h}/2}^\infty\frac{\abs{\sin(t)}}{t^2}\,dt\notag\\
 &\leq  2\abs{s}\abs{h}\left(\int_{\abs{h}/2}^1\frac{dt}{t}+\int_{1}^\infty
\frac{dt}{t^2}\right)
 =  2\abs{s}\abs{h}\left(\abs{\log\abs{h}}+1+\log(2)\right).\label{vs4}
\end{align}
For the final term in \eqref{vs2},
\begin{align}
&\int_{\abs{t}>1}\left|\frac{2e^{-ist}(1-e^{-iht})}{t^3}\right|\,dt
=8\int_1^\infty\frac{\abs{\sin(ht/2)}}{t^3}\,dt
=2h^2\int_{\abs{h}/2}^\infty\frac{\abs{\sin(t)}}{t^3}\,dt\notag\\
&\leq 2h^2\left(\int_{\abs{h}/2}^1\frac{dt}{t^2}+\int_1^\infty\frac{dt}{t^3}\right)
\leq 4\abs{h}.\label{vs5}
\end{align}
Combining \eqref{vsholder}, \eqref{vs0}, \eqref{vs3}, \eqref{vs4} and \eqref{vs5} 
completes the proof.

(d)  Let $\psi\fn(0,1)\to(0,\infty)$ such that
$\psi(h)=o(h\log(h))$ as $h\to 0^+$ and
$\psi(h)\geq h\abs{\log(h)}^{1/2}$. Define the family of linear operators 
$M_{s,h}\fn\alexc\to\C$ by
$M_{s,h}[f]=(\Omega_{f}(s+h)-\Omega_f(s))/\psi(h)$.  The estimate in (c) shows that,
for each $s>0$ and $0<h<1/e$, $M_{s,h}$ is a bounded linear operator.  Define
$F_{s,h}(t)=\omega_h(t)e^{-i\theta_{s,h}(t)}$ where
$$
\omega_{h}(t)=\left\{
\begin{array}{cl}
0, & t\leq 1\\
t-1, & 1\leq t\leq 2\\
1, & 2\leq t\leq 2/h\\
1+2/h-t, & 2/h\leq t\leq 1+2/h\\
0, & t\geq 1+2/h
\end{array}
\right.
$$ 
and (c.f. \eqref{vs4}) $\theta_{s,h}(t)={\rm arg}(ie^{-ist}(1-e^{-iht}))$.  Let $f_{s,h}=F_{s,h}'$.  Then
$f_{s,h}\in\alexc$ and $\norm{f_{s,h}}=1$.  Using \eqref{vs3}, \eqref{vs4}, 
\eqref{vs5}, this gives
\begin{eqnarray*}
\frac{\abs{M_{s,h}[f_{s,h}]}}{\norm{f_{s,h}}} & = & 
\frac{1}{\psi(h)}\left|\intinf(v_{s+h}'(t)-v_s'(t))F_{s,h}(t)\,dt\right|\\
 & = & \frac{1}{\psi(h)}\left|\int_{1}^{1+2/h}
(v_{s+h}'(t)-v_s'(t))\omega_h(t)e^{-i\theta_{s,h}(t)}\,dt\right|\\
 & \geq & \frac{2s}{\psi(h)}\int_2^{2/h}\frac{\abs{\sin(ht/2)}}{t^2}\,dt
-\frac{4h}{\psi(h)}.
\end{eqnarray*}
To get a lower bound on the logarithm term, notice that
$\abs{\sin(t)}\geq 2t/\pi$ for $\abs{t}\leq\pi/2$.
As with \eqref{vs4} we get
$$
\frac{\abs{M_{s,h}[f_{s,h}]}}{\norm{f_{s,h}}}  \geq
\frac{2sh\abs{\log(h)}}{\pi\psi(h)}-\frac{4h}{\psi(h)}
\to  \infty \text{ as } h\to 0^+.
$$
Therefore, the logarithm term gives a sharp estimate as $h\to 0$ and $\Omega_f$ need not
be Lipschitz continuous.

(e) Notice that $u_s(t)=su_1(st)$ and $u_s\in L^\rho(\R)$ for each $1<\rho\leq\infty$.
Since $v_{s+h}(t)-v_s(t)=v_h(t)+u_s(t)u_h(t)$ the result follows by H\"older's inequality.

\end{proof}

\begin{remark}\label{remarkab}
It is only necessary to define $\Omega_f$ to the addition of a linear function.
If $a,b\in\R$ then
$$
w_s(t)\defeq\int_{a}^s\int_b^{\sigma_2}e^{-i\sigma_1t}\,d\sigma_1\,d\sigma_2
  =  e^{-ibt}\left[v_{s-b}(t)-v_{a-b}(t)\right]
  =  \frac{-i(s-a)te^{-ibt}-e^{-ist}+e^{-iat}}{t^2}.
$$
And, $\intinf w_sf$ differs from $\Omega_f(s)$ by a linear function.
However, $w_s$ is of finite variation if and only if $b=0$.
In the definition of $v_s$ there is nothing to be gained by using lower limits
of integration other than $0$.
\end{remark}

\begin{remark}
A calculation shows that
$$
v_{s+h}(t)-2v_s(t)+v_{s-h}(t)=h^2e^{-ist}g(ht/2),
$$
where $g(t)=\sin^2(t)/t^2$.  Then the second central difference is
\begin{equation}
\frac{\Omega_f(s+h)-2\Omega_f(s)+\Omega_f(s-h)}{h^2}=\intinf e^{-ist}\left[\frac{\sin(ht/2)}{ht/2}
\right]^2f(t)\,dt.\label{fejer1}
\end{equation}
If $f\in L^1(\R)$, taking the limit $h\to0$ with dominated convergence 
gives $\fhat(s)$ in terms of the Fej\'er kernel.  See Example~\ref{examplekernels}.

Note that the variation of $t\mapsto e^{-ist}\left[\frac{\sin(ht/2)}{ht/2}
\right]^2$ is bounded above by $2\abs{s}\norm{g}_1/\abs{h}+{\rm var}(g)$.  The
limit in \eqref{fejer1} need not exist but by \cite[Theorem~22]{talviladenjoy}
\begin{equation}
\lim_{h\to0}\frac{\Omega_f(s+h)-2\Omega_f(s)+\Omega_f(s-h)}{h}=
\intinf\lim_{h\to0}he^{-ist}\left[\frac{\sin(ht/2)}{ht/2}
\right]^2f(t)\,dt=0.\label{fejer2}
\end{equation}
\end{remark}

\section{Banach spaces}\label{sectionbanach}
Banach spaces of integrable distributions, their primitives and Fourier transforms
are defined.

\begin{theorem}[Banach spaces]\label{theoremBanachspaces}
(a)  With $\Omega_f$ given in Theorem~\ref{theoremOmegaf},
define $\Cc=\{\Omega_f\mid f\in\alexc\}$.  If $f\in\alexc$ define the norm
$\norm{\Omega_f}^{(1)}=\norm{f}$.   Then $\Cc$ is a Banach space isometrically
isomorphic to $\alexc$.

(b) Define $\Dc=\{\Omega_f''\mid f\in\alexc\}$.  If $f\in\alexc$ define
$\norm{\Omega_f''}^{(2)}=\norm{f}$.   Then $\Dc$ is a Banach space isometrically
isomorphic to $\alexc$.

(c) The four Banach spaces $\alexc$, $\balexc$, $\Cc$ and $\Dc$ are isometrically isomorphic.
\end{theorem}

\begin{proof}
(a) Define $\Lambda\fn\alexc\to\Cc$ by $\Lambda(f)=\Omega_f$. By definition, 
$\Lambda$ is linear and onto $\Cc$.  Prove it is one-to-one.  Suppose $f\in\alexc$ and
$\Lambda(f)=0$, i.e., $\Omega_f(s)=0$ for all $s\in\R$.  For each $\phi\in\Sc$ then,
$$
0=\intinf\Omega_f(s)\phi''(s)\,ds=\intinf\intinf f(t)v_s(t)\phi''(s)\,dt\,ds.
$$
To use the Fubini theorem for $\alexc$, Theorem~\ref{theoremfubini},
note that from Theorem~\ref{theoremOmegaf}(a), ${\rm var}(v_s\phi''(s))=cs^2\abs{\phi''(s)}$ and,
as a function of $s$, this is in $L^1(\R)$.  As well then,
$\abs{v_s(t)\phi''(s)}\leq cs^2\abs{\phi''(s)}$.
Condition (iii) of Theorem~\ref{theoremfubini} is thus satisfied.  And,
$$
0=\intinf f(t)\intinf v_s(t)\phi''(s)\,ds\,dt.
$$
Integrating by parts twice, and using the fact that $\partial^2 v_s(t)/\partial s^2=e^{-ist}$,
shows that $\intinf f\phihat=0$.

The Fourier transform is an isomorphism on $\Sc$.  The transform and its inverse
are sequentially continuous so then $\langle f,\phi\rangle=0$ for all $\phi\in\Sc$ and
$f=0$ as a tempered distribution.

Hence, $\Lambda$ is a linear isometry and isomorphism, and
$\alexc$ and $\Cc$ are isometrically isomorphic
Banach spaces.

(b) The second distributional derivative provides a linear isometry and isomorphism
from $\Cc$ to $\Dc$. The estimate $\Omega_f(s)\sim F(\infty)s^2/2$ as $s\to0$ in 
Theorem~\ref{theoremOmegaf}(b)
shows that the only linear function in $\Cc$ is the zero function.

(c)
It was shown in \cite{talviladenjoy} that $\alexc$ and $\balexc$ are isometrically
isomorphic.
\end{proof}

\section{The Fourier transform in $\alexc$}\label{sectionFTalexc}

Now we can define the Fourier transform of distributions in $\alexc$.
\begin{defn}\label{defnfhatalexc}
Let $f\in\alexc$. 
Define the Fourier transform as $\hat{\mbox{}}\,\fn\alexc\to\Dc$ by $\fhat=\Omega_f''$,
i.e., 
$\langle \fhat,\phi\rangle=\langle \Omega_f,\phi''\rangle=\intinf\Omega_f(s)\phi''(s)\,ds$ 
for each $\phi\in\Sc$.
\end{defn}
The space of Fourier transforms on $\alexc$ is then denoted $\Dc$, as per 
Theorem~\ref{theoremBanachspaces}.

The Fourier transform
of a tempered distribution, $T\in\Scp$, is $\langle \hatT,\phi\rangle=\langle T, \phihat\rangle$
for each Schwartz function $\phi\in\Sc$.
Since $\Omega_f$ is continuous and of slow (polynomial) growth then $\fhat$ as defined
by Definition~\ref{defnfhatalexc} is a tempered distribution.

\begin{theorem}\label{theoremequivalentdefns}
Definition \ref{defnfhatalexc} agrees with the definition of the Fourier transform of a tempered
distribution.
\end{theorem}
\begin{proof}
This is contained in the proof of Theorem~\ref{theoremBanachspaces}(a).
\end{proof}
Hence, Fourier transforms in 
$\alexc$ inherit the properties of Fourier transforms in $\Scp$.
See, for example, \cite[\S30]{donoghue} or \cite[2.3.22]{grafakosclassical}.

\begin{remark}\label{remarkLp}
The existence of $\Omega_f$ depends on the fact that $v_s$ is of bounded variation
and hence is in the dual space of $\alexc$.
If the Fourier transform of a function $f\in L^1(\R)$ is integrated once, then in
place of \eqref{Omega1} and \eqref{vs}, one obtains
$$
\Psif(s)\defeq\int_0^s\fhat(\sigma)\,d\sigma=\intinf u_s(t)f(t)\,dt,
$$
where $u_s(t)=\partial v_s(t)/\partial s=(1-e^{-ist})/(it)$.  The Fourier transform
is then given by the pointwise derivative $\fhat(s)=\Psif'(s)$ for each $s\in\R$.
Since $u_s\in L^q(\R)$ for each $1< q\leq\infty$ then $\Psif$ also exists
when $f\in L^p(\R)$ for some $1\leq p<\infty$. And the Fourier transform can be
defined as the distributional derivative $\fhat=\Psif'$.  This method was used to define
Fourier transforms in \cite{talvilaLpFourier}, where corresponding results were
obtained as in the current paper.

Let $\phi\in\Sc$.
If $f\in L^p(\R)$ for some $1\leq p<\infty$ then, using the estimates in the proof
of Theorem~\ref{theoremBanachspaces},
the Fubini--Tonelli theorem
and integration by parts show
\begin{eqnarray}
\langle\Psif',\phi\rangle & = & -\intinf\intinf u_s(t)f(t)\,dt\phi'(s)\,ds
 = -\intinf f(t)\intinf \frac{\partial v_s(t)}{\partial s}\phi'(s)\,ds\,dt\notag\\
 & = &   \intinf f(t)\intinf v_s(t)\phi''(s)
\,ds\,dt.\label{Psif}
\end{eqnarray}
The Fubini--Tonelli theorem then shows $\langle\Omega_f'',\phi\rangle$ is given
by the same integral in \eqref{Psif}.  The two methods then agree.
\end{remark}

\section{Exchange formula}\label{sectionexchange}

The exchange formula is the key to inversion and convolution results in following
sections.

Denote the functions of bounded variation by $\bv$.
Note that functions of bounded variation are bounded and regulated on $\Rbar$,
i.e., they have left limits in $(-\infty,\infty]$ and right limits in $[-\infty,\infty)$.
See, for example, \cite{appell}.
\begin{theorem}[Exchange formula]\label{theoremexchange}
Let $f\in\alexc$.  Write $p_n(x)=x^n$. Define $g(s)=s^{-2}\int_0^sh(t)\,dt$ for a
function $h\in\bv$ that is right continuous such that $h/p_1\in L^1(\R)$.
Then,\\
(a) $g$ is Lipschitz continuous on compact sets in $\R\setminus\{0\}$,\\
(b) $g\in L^1(\R)$,\\
(c) $p_1g\in\bv$, $p_1g(s)=o(1)$ as $s\to 0$ or as $\abs{s}\to\infty$,\\
(d) $p_2g'\in\bv$, $p_2g'(s)=o(1)$ as $s\to0$ and 
$\lim_{s\to\pm\infty}p_2g'(s)=
\lim_{s\to\pm\infty}h(s)$.\\
(e) The integral $\intinf \fhat g$ exists and
\begin{eqnarray*}
\intinf \fhat g & = & \intinf\Omega_f(s)\,dg'(s)=
\intinf\frac{\Omega_f(s)}{s^2}\left(dh(s)-\frac{4h(s)\,ds}{s}
+\frac{6}{s^2}\int_0^sh(t)\,dt\,ds\right)\\
 & = &
\intinf\frac{\Omega_f(s)}{s^2}(d(p_2g')(s)-2d(p_1g)(s)+2g(s)\,ds).
\end{eqnarray*}
There are the estimates
\begin{eqnarray*}
\left|\intinf \fhat g\right| & \leq & c\norm{f}\intinf s^2\abs{dg'(s)},\\
\left|\intinf \fhat g\right| & \leq & c\norm{f}({\rm var}(h)+10\norm{h/p_1}_1),\\
\abs{\intinf\fhat g} & \leq & c\norm{f}({\rm var}(p_2g')+2{\rm var}(p_1g) +2\norm{g}_1),
\end{eqnarray*}
where $c$ is given in Theorem~\ref{theoremOmegaf}.\\
(f) The exchange formula is $\intinf \fhat g=\intinf f\ghat$ and
$\ghat\in\bv\cap C(\R)$ so that $\left|\intinf \fhat g\right|\leq \norm{f}{\rm var}(\ghat)$.
\end{theorem}
The proof uses Stieltjes integrals (for example,
\cite{mcleod} or \cite{monteiroslaviktvrdy}) and distributional integrals (\cite{talviladenjoy}, \cite{talvilaacrn}).

Different forms of the hypotheses are given in Corollary~\ref{corollaryexchange}
and Corollary~\ref{corollarytauberian}.

\begin{proof}
(a) This follows since $h$ is bounded.

(b) By
the Fubini--Tonelli theorem, 
\begin{eqnarray*}
\norm{g}_1 & \leq & \int_0^\infty\int_0^s\abs{h(t)}\,dt\,\frac{ds}{s^2}
+ \int_{-\infty}^0\int_s^0\abs{h(t)}\,dt\,\frac{ds}{s^2}\\
 & = & \int_0^\infty\abs{h(t)}\int_t^\infty\frac{ds}{s^2}\,dt
+\int_{-\infty}^0\abs{h(t)}\int_{-\infty}^t\frac{ds}{s^2}\,dt =\norm{h/p_1}_1.
\end{eqnarray*}

(c) Due to (a) and the Fubini--Tonelli theorem,
$$
{\rm var}(p_1g)  =  \intinf\left|-\frac{1}{s^2}\int_0^sh(t)\,dt+\frac{h(s)}{s}\right|\,ds
  \leq  2\norm{h/p_1}_1.
$$
Since $p_1g$ is of bounded variation,
we can then write $\lim_{s\to0^+}sg(s)=\alpha\in\R$.
If $\alpha>0$ then there is $\delta>0$ such that for all $0<s<\delta$ we have
$g(s)>\alpha/(2s)$.  But then $g\not\in L^1(\R)$.  Similarly, if $\alpha<0$.
Hence, $\lim_{s\to 0}sg(s)=0$.

Since $g\in L^1(\R)$ and $\lim_{s\to\infty}sg(s)$ exists we must have $p_1g(s)\to0$ as
$\abs{s}\to\infty$.

(d) Let $H(s)=\int_0^sh(t)\,dt$.  Since $p_2g'=-2p_1g+h$ almost everywhere, this gives 
${\rm var}(p_2g')\leq 4\norm{h/p_1}_1+{\rm essvar}(h)$.  But with $h$ right continuous,
the left and right derivatives of $g$ exist everywhere except perhaps at the origin:
\begin{eqnarray*}
\partial_+g(s) & = & \lim_{x\to 0^+}\frac{g(s+x)-g(s)}{x}=-\frac{2H(s)}{s^3}+\frac{h(s+)}{s^2}
=-\frac{2H(s)}{s^3} + \frac{h(s)}{s^2}\\
\partial_-g(s) & = & \lim_{x\to 0^-}\frac{g(s+x)-g(s)}{x}=-\frac{2H(s)}{s^3}+\frac{h(s-)}{s^2}.
\end{eqnarray*}
This suffices for computing the variation of functions containing $g'$ or $h$ or in the Stieltjes
integrals below that contain $dg'$ or $dh$.  As such,
${\rm essvar}(h)={\rm var}(h)$.

Since $h\in\bv$ and $h/p_1\in L^1(\R)$ it follows that $h(s)=o(1)$ as $s\to 0$.
Then $H(s)=o(s)$ as $s\to 0$.  This shows $p_2g'(s)=o(1)$ as $s\to0$.

And, $\lim_{s\to\pm\infty}p_2g'(s)=-2\lim_{s\to\pm\infty}p_1g(s)+\lim_{s\to\pm\infty}h(s)=
\lim_{s\to\pm\infty}h(s)$.

(e) Write $\Omega_f=p_2A$.  From Theorem~\ref{theoremOmegaf}, $A(s)=o(1)$ as $\abs{s}\to\infty$
and $A\in\balexc$.  

By the usual properties of distributional
derivatives, $\Omega_f''=p_2A''+4p_1A'+2A$.  Write $\intinf \fhat g=
I_1+4I_2+2I_3$ where
$$
I_1=\intinf A''p_2g,\quad I_2=\intinf A'p_1g,\quad I_3=\intinf A(s)g(s)\,ds.
$$

Integral $I_1$ exists as a second order distributional integral, $A''\in\alexx$ 
\cite[\S2]{talvilaacrn}, since $p_2g$ has a primitive that is of bounded
variation.
Since $h\in\bv$ it is bounded, so
$I_1=-\intinf A'h=\intinf A(s)\,dh(s)$ by Definition~2.6 in \cite{talvilaacrn}
and integration by parts in $\alexc$ \eqref{parts}.

Integral $I_2$ exists in $\alexc$ since $p_1g$ is of bounded variation.

Integral $I_3$ exists in $L^1(\R)$ since $g\in L^1(\R)$.

Since $p_1g$ is bounded, we can
integrate by parts \eqref{parts}
to write 
$I_2=\intinf A'H/p_1=-\intinf A(s)\,d(H(s)/s)$.

Now let $0<x<y<\infty$.  Using Proposition~\ref{prop6.2} and Proposition~\ref{prop6.3},
$$
-\int_0^\infty A'p_1g  =  \lim_{\substack{y\to\infty\\x\to 0^+}}\int_x^yA(s)(
-h(s)/s+H(s)/s^2)\,ds.
$$

Consider the integrands in $I_1+4I_2+2I_3$ on $(0,\infty)$. 
From the above results this gives
$$
\lim_{\substack{y\to\infty\\x\to 0^+}}\int_x^yA(s)\left(dh(s)-\frac{4h(s)\,ds}{s}
+\frac{6H(s)\,ds}{s^2}\right).
$$
Notice that integrals in $A''\in\alexx$ need not exist on subintervals so it is 
necessary to rewrite $\intinf A''p_2g$ as an integral in $\alexc$ before considering
limits over subintervals.

And, again using Proposition~\ref{prop6.2} and Proposition~\ref{prop6.3},
\begin{eqnarray}
\int_0^\infty A(s)s^2\,dg'(s) & = & \lim_{\substack{y\to\infty\\x\to 0^+}}\int_x^y
A(s)s^2\,d\left(\frac{h(s)}{s^2}-\frac{2H(s)}{s^3}\right)\notag\\
 & = & \lim_{\substack{y\to\infty\\x\to 0^+}}\int_x^y
A(s)\left(dh(s)-\frac{4h(s)\,ds}{s}+\frac{6H(s)\,ds}{s^2}\right).\label{fhatgestimateh}
\end{eqnarray}

Similarly, as
$x\to 0^-$ and $y\to-\infty$.  Using the Hake theorem for
$\alexc$ \cite[Theorem~25]{talviladenjoy},
$\intinf \fhat g=\intinf\Omega_f''g=\intinf A(s)s^2\,dg'(s)=\intinf\Omega_f(s)\,dg'(s)$.

From \ref{fhatgestimateh},
$$
\intinf\fhat g=\intinf\frac{\Omega_f(s)}{s^2}\left(dh(s)-\frac{4h(s)\,ds}{s}
+\frac{6H(s)\,ds}{s^2}\right),
$$
from which $\abs{\intinf\fhat g}\leq c\norm{f}({\rm var}(h)+10\norm{h/p_1}_1)$.
Similarly,
$$
\intinf\fhat g=\intinf\frac{\Omega_f(s)}{s^2}(d(p_2g')(s)-2d(p_1g)(s)+2g(s)\,ds),
$$
from which $\abs{\intinf\fhat g}\leq c\norm{f}({\rm var}(p_2g')+2{\rm var}(p_1g) +2\norm{g}_1)$.

For each $0<x<y$, the results above show that 
$\int_x^ys^2dg'(s)=\int_x^yd(p_2g'(s))-2\int_x^yd(p_1g(s))+2\int_x^yg(s)\,ds$.
The properties of $g$, $p_1g$ and $p_2g'$ in (b), (c) and (d) then show
$\intinf s^2dg'(s)=\intinf d(p_2g'(s))-2\intinf d(p_1g(s))+2\intinf g(s)\,ds$.
The three integrals on the right converge absolutely so the integral on the
left also converges absolutely.

From Theorem~\ref{theoremOmegaf}, ${\rm var}(v_s)=cs^2$.  Note that
$$
\norm{v_s}_\infty=s^2\max_{t\geq 0}\frac{\sqrt{4\sin^2(t/2)+(t-\sin(t))^2}}{t^2}=s^2/2.
$$
Since $\intinf s^2\abs{dg'(s)}$ exists,
a minor modification of Theorem~\ref{theoremfubini} now gives
(inserting a variable in $f$ and $v_s$ to avoid confusion in the iterated integrals)
$$
\intinf \fhat g  =  \intinf\intinf v_s(t)f(t)\,dt\,dg'(s)
  =  \intinf f(t)\intinf v_s(t)\,dg'(s)\,dt.
$$
Since $\partial^2 v_s(t)/\partial s^2=e^{-ist}$,
integrating by parts twice and using (c) and (d), now shows that $\intinf \fhat g=\intinf f\ghat$.

(f) Since the estimates in part (e) holds
for all $f\in\alexc$ then $\ghat$ is in the dual
space of $\alexc$, which is  the functions of essential bounded variation
\cite[Corollary~9]{talviladenjoy}.  Since $g\in L^1(\R)$  then $\ghat$ is continuous,
of bounded variation, and vanishes at $\pm\infty$.
The H\"older inequality \eqref{holder} gives the final estimate.
\end{proof}

\begin{corollary}\label{corollaryexchange}
Let $f\in\alexc$.  Let $g\in L^1(\R)$ such that
$g$ is absolutely continuous on compact intervals in $\R\setminus\{0\}$,
$p_1g\in\bv$ and $p_2g'\in\bv$.
Then $\ghat\in\bv\cap C(\R)$ and
$\intinf \fhat g=\intinf f \ghat$.
\end{corollary}
\begin{proof}
Define $h(s)=s^2g'(s)+2sg(s)$.  Then $h\in\bv$ and $g(s)=s^{-2}\int_0^sh(t)\,dt$.
Hence, $g$ is absolutely continuous except perhaps at the origin.
As well,
$$
\norm{h/p_1}_1=\intinf\abs{sg'(s)+2g(s)}\,ds=\intinf\abs{(p_1g)'(s)+g(s)}\,ds\leq
{\rm var}(p_1g)+\norm{g}_1.
$$
\end{proof}

\begin{remark}\label{remarktranslate}
The conditions in Theorem~\ref{theoremexchange}, Corollary~\ref{corollaryexchange} 
and Corollary~\ref{corollarytauberian} can be translated.
Let $a\in\R$.  The translation is written $\tau_ag(x)=g(x-a)$.  Let $e_x(s)=e^{xs}$.  
If the exchange identity is written as
$$
\intinf(\tau_a\fhat)g=\intinf \fhat(\tau_{-a}g)=\intinf fe_{ia}\ghat$$
then the conditions in Theorem~\ref{theoremexchange} are applied at $a$ 
instead of at $0$.  For example, (c) becomes
$x\mapsto (x-a)g(x)\in\bv$ and $(x-a)g(x)=o(1)$ as $x\to a$.
Then $x\mapsto e^{iax}\ghat(x)\in\bv$.

And, $$
\intinf\fhat e_{ia}g=\intinf f(\tau_a g),$$
provided $e_{ia}g$ satisfies the hypotheses of Corollary~\ref{corollarytauberian}.
It then follows that $\tau_ag$ and hence $g$ are in $\bv$.
\end{remark}

The hypotheses of Theorem~\ref{theoremexchange} can be rewritten as
integral conditions that ensure a Fourier transform is of bounded variation.
\begin{corollary}\label{corollarytauberian}
Let $g\in L^1(\R)$ such that for each $\eta>0$, $g$ is Lipschitz
continuous on $\R\setminus(-\eta,\eta)$ and $g'$ is equivalent to a function of 
bounded variation on $\R\setminus(-\eta,\eta)$.
Suppose $\intinf\abs{s}\abs{g'(s)}\,ds<\infty$ and
$\intinf s^2\abs{dg'(s)}<\infty$.
Then the conclusions of Theorem~\ref{theoremexchange} hold.
\end{corollary}
\begin{proof}
First note that $\intinf\abs{s}\abs{g'(s)}\,ds=\intinf\abs{s}\abs{dg(s)}$,
by Proposition~\ref{prop6.2} and Proposition~\ref{prop6.3}.

Let $\epsilon>0$.  Let $0<x<y<\infty$.  The Riemann--Stieltjes integral 
$\int_x^y\abs{s}\abs{dg(s)}$ and the Riemann integral $\int_x^y\abs{g(s)}\,ds$ exist.
There is then $\delta>0$ such that for any partition 
$x=x_0<x_1<\ldots<x_N=y$ satisfying $\max_{1\leq n\leq N}(x_n-x_{n-1})<\delta$ we have
$$
\abs{\sum_{n=1}^N\abs{z_n}\abs{g(x_{n})-g(x_{n-1})}-\int_x^y\abs{s}\abs{dg(s)}}  <  \epsilon
$$
and
$$
\abs{\sum_{n=1}^N\abs{g(z_n)}(x_{n}-x_{n-1})-\int_x^y\abs{g(s)}\,ds}  <  \epsilon,
$$
for each tag $z_n\in[x_{n-1},x_n]$.

Now consider
\begin{eqnarray*}
\sum_{n=1}^N\abs{p_1g(x_n)-p_1g(x_{n-1})} & \leq & 
\sum_{n=1}^N\abs{x_{n}}\abs{g(x_n)-g(x_{n-1})} + \sum_{n=1}^N\abs{g(x_{n-1})}(x_n-x_{n-1})\\
 & < & \int_0^\infty\abs{s}\abs{dg(s)}+\int_0^\infty\abs{g(s)}\,ds+2\epsilon.
\end{eqnarray*}
Taking limits $x\to0^+$ and $y\to\infty$ shows
$p_1g$ is of bounded variation on $[0,\infty)$.  Similarly on $(-\infty,0]$.

In a similar manner, there is the inequality
\begin{eqnarray*}
\sum_{n=1}^N\abs{p_2g'(x_n)-p_2g'(x_{n-1})} & \leq &
\sum_{n=1}^N x_{n}^2\abs{g'(x_{n})-g'(x_{n-1})}
+\sum_{n=1}^N
\abs{g'(x_{n-1})}(x_n^2-x_{n-1}^2)\\
 & < & \int_0^\infty s^2\abs{dg'(s)}+\int_0^\infty\abs{g'(s)}\,d(s^2)+2\epsilon.
\end{eqnarray*}
Using limits over subintervals as above, Proposition~\ref{prop6.3} shows
$$
\abs{\intinf \abs{g'(s)}\,d(s^2)}=2\abs{\intinf \abs{g'(s)}s\,ds}
\leq 2\intinf \abs{g'(s)}\abs{s}\,ds.
$$  
And,
$p_2g'\in\bv$.

Now define $h(s)=s^2g'(s+)+2sg(s+)$.  Then $h$ is right continuous and of bounded variation.
And, $\norm{h/p_1}_1\leq\norm{p_1g_1'}_1+2\norm{g}_1$.  The hypotheses of
Theorem~\ref{theoremexchange} are then satisfied.
\end{proof}
The results of Corollary~\ref{corollarytauberian} can be translated as with 
Remark~\ref{remarktranslate}.

\begin{remark}\label{remarkghatbvconditions}
Any condition giving $\ghat\in\bv$ that requires $g\in L^1(\R)$ cannot be necessary
since then $\ghat$ is continuous. 
Also, the proof of Theorem~\ref{theoremexchange} assumes $A\in\balexc$, i.e., continuous,
but from
Theorem~\ref{theoremOmegaf} it is known that $A$ is H\"older continuous.  This
extra smoothness in $A$ means 
the smoothness conditions on $g$ are not necessary.
See the following example.

Other sufficient conditions that ensure $\ghat\in\bv\cap C^1(\R)$ are given in
\cite[Proposition~7.2]{talvilaLpFourier}.
The simple condition that $g\in C^2(\R)$ such that $x^mg^{(n)}(x)$ is bounded
on $\R$ for all integers $0\leq m\leq 4$ and $0\leq n\leq 2$ suffices.  This
condition is also translation invariant.  The Fourier transform
in $\alexc$ can then be defined as $\langle\fhat,g\rangle=\langle f,\ghat\rangle$ for
such functions.

A similar
result showed when the Fourier transform of a function was in $L^q(\R)$ for
some $1<q\leq\infty$ \cite[Theorem~6.1]{talvilaLpFourier}.
\end{remark}

\begin{example}\label{exampleexchange}
(a) Let $g$ be given by the modified Bessel function $K_0$, extended as an even function
to $(-\infty,0)$.  Since
$$
K_0(x)\sim\left\{\begin{array}{cl}
-\log(x), & x\to 0^+\\
\sqrt{\frac{\pi}{2x}}e^{-x}, & x\to\infty,
\end{array}
\right.
$$
then $g$
satisfies the hypotheses of 
Corollary~\ref{corollarytauberian} and gives 
$\ghat(s)=\pi/\sqrt{s^2+1}$ \cite[1.12(40)]{erdelyi}.  Then
$\ghat$ is of bounded variation but $g$ itself is not bounded and hence
not of bounded variation.

(b) Let $g(x)=(1-x^2)^{\nu-1/2}\chi_{(-1,1)}(x)$.  Then
$\ghat(s)=\sqrt{\pi}\,2^\nu\Gamma(\nu+1/2)\abs{s}^{-\nu}J_\nu(\abs{s})$ \cite[1.3(8)]{erdelyi}, for
$\nu>-1/2$.  Since $g$ is of compact support then $\ghat$ is analytic.
The asymptotic formula
$\ghat(s)\sim 2^{\nu+1/2}\Gamma(\nu+1/2)s^{-(\nu+1/2)}\cos(s-\nu\pi/2-\pi/4)$
as $s\to\infty$ shows $\ghat\in\bv$ when $\nu>1/2$.  However, the conditions of
Theorem~\ref{theoremexchange} and Corollaries \ref{corollaryexchange} and
\ref{corollarytauberian} are satisfied exactly when $\nu\geq 3/2$.
Note that in the case $\nu=3/2$, $g'$ can have jump discontinuities, as discussed
in the proof of Theorem~\ref{theoremexchange}(d).
\end{example}

\section{Inversion}\label{sectioninversion}

The exchange formula leads to inversion in the Alexiewicz norm using a summability kernel.
Pointwise inversion follows for $L^p(\R)$ functions ($1<p<\infty$) by twice
differentiating $\Omega_{\fhat}$. Recall Definition~\ref{defnpn}, $p_n(x)=x^n$.

\begin{theorem}[Inversion in $\alexc$]\label{theoreminversion}
Let $f\in\alexc$.
Let $\psi\in\Lone$ such that $\intinf\psi(t)\,dt=1$,
$\psihat$ is absolutely continuous, $p_2\psihat\in L^1(\R)$,
$p_2\psihat\,'\in L^1(\R)$ and $p_2\psihat\,'\in\bv$.
For each $a>0$ define $\psi_a(x)=\psi(x/a)/a$.  Define a family of kernels $K_a(s)=\widehat{\psi_a}(s)$.
Let $e_x(s)=e^{xs}$ and $I_a[f](x)=(1/(2\pi))\intinf e_{ix}K_a\fhat$.
Then $\lim_{a\to0^+}\norm{f-I_a[f]}=0$.
\end{theorem}
\begin{proof}
The conditions on $\psi$ also hold for $\psi_a$. For fixed $x\in\R$ and $a>0$, write
$g(s)=e^{ixs}K_a(s)$.  Show the conditions of
Corollary~\ref{corollarytauberian} hold for $g$.  We have
\begin{eqnarray*}
\intinf \abs{s}\abs{g'(s)}\,ds & \leq & \abs{x}\intinf\abs{s}\abs{\psiahat(s)}\,ds
+\intinf\abs{s}\abs{\psiahat\,'(s)}\,ds.
\end{eqnarray*}
From 
Proposition~\ref{prop6.2} and Proposition~\ref{prop6.3},
\begin{eqnarray*}
\intinf s^2\abs{dg'(s)} & \leq & \abs{x}\intinf s^2\abs{d(e^{ixs}\psiahat(s))}
+\intinf s^2\abs{d(e^{ixs}\psiahat\,'(s))}\\
 & \leq & x^2\intinf s^2\abs{\psiahat(s)}\,ds+2\abs{x}\intinf s^2\abs{\psiahat\,'(s)}\,ds
+\intinf s^2\abs{d\psiahat\,'(s)}.
\end{eqnarray*}

Since $\psi$ and $\psihat$ are in $L^1(\R)$ the
pointwise Fourier inversion theorem applies to $\psi$ \eqref{pointwiseinversion}.
From Theorem~\ref{theoremexchange}(f) then $I_a[f](x)
=f\ast\psi_a(x)$.
By approximation with convolutions in norm 
\cite[Theorem~3.4]{talvilaconvolution} the conclusion follows.
\end{proof}

A similar inversion theorem appears in \cite{talvilaLpFourier} for $L^p(\R)$ functions
and $1<p<\infty$.

\begin{example}\label{examplekernels}
The three commonly used summability kernels:
Ces\`aro--Fej\'er ($K_a(s)=(1-a\abs{s})\chi_{[-1/a,1/a]}(s)/(2\pi)$ with 
${\widehat K_a}(t)=2a\sin^2[t/(2a)]/(\pi t^2)$),
Abel--Poisson ($K_a(s)=e^{-a\abs{s}}/\pi$ with ${\widehat K_a}(t)=2a/[\pi(t^2+a^2)]$) and
Gauss--Weierstrass ($K_a(s)=e^{-a^2s^2}/(2\pi)$ with ${\widehat K_a}(t)=e^{-t^2/(4a^2)}/(2\sqrt{\pi}a)$)
satisfy the conditions of Theorem~\ref{theoreminversion}.
While the Dirichlet kernel ($K_a(s)=\chi_{[-1/a,1/a]}(s)/(2\pi)$ with ${\widehat K_a}(t)=\sin(t/a)/(\pi t)$)
does not.

The way summability kernels are defined in Theorem~\ref{theoreminversion} means
the inverses of these kernels are associated with the given names.
\end{example}

If $f$ and $\fhat$ are in $L^1(\R)$ then the pointwise inversion theorem reads
\begin{equation}
f(x)=\frac{1}{2\pi}{\hat \fhat}(-x)=\frac{1}{2\pi}\intinf e^{isx}\fhat(s)\,ds,\label{pointwiseinversion}
\end{equation}
for each $x\in\R$.  For example, \cite[Theorem~8.26]{folland}.  When $f\in L^p(\R)$ for some $1\leq p<\infty$ the Fourier
transform can be defined as $\fhat=\Psif'$, where $\Psif$ is given in Remark~\ref{remarkLp}.
An inversion analogous to the $L^1(\R)$ case above then holds.

\begin{theorem}[Inversion for $L^p(\R)$ functions]\label{theoremLpinversion}
Let $f\in L^p(\R)$ for some $1< p<\infty$.  Let $\Psif(s)=\intinf u_s(t)f(t)\,dt$,
where $u_s(t)=(1-e^{-ist})/(it)$, and define $\fhat=\Psif'$.  Then
\begin{equation}
\frac{1}{2\pi}\frac{d^2}{dx^2}\Omega_{\fhat}(-x)=f(x)\label{dx^2}
\end{equation}
for almost all $x\in\R$.
If $f$ is continuous at $x\in\R$ then \eqref{dx^2} holds at $x$.
If $f$ is bounded in a neighbourhood of $x$ and approximately continuous at $x$
then \eqref{dx^2} holds at $x$.
\end{theorem}

Note that if $2<p<\infty$ then the Fourier transform is a tempered distribution
that need not have any pointwise values.  For example,  \cite[4.13, p.~34]{steinweiss}.
However, formula \eqref{dx^2} returns pointwise values almost everywhere.

\begin{proof}
The function $v_{-x}$ satisfies the conditions of the exchange theorem in 
\cite{talvilaLpFourier} (Theorem~6.1).
Therefore,
$$
\Omega_{\fhat}(-x)=\intinf v_{-x}\fhat=
\intinf {\hat v_{-x}}(y)f(y)\,dy.
$$
With notation from Theorem~\ref{theoremexchange}, 
\begin{eqnarray*}
{\hat v_{-x}}(y) & = & \intinf e^{-iyt}\left(\frac{1+ixt-e^{ixt}}{t^2}\right)dt
={\hat p_{-2}}(y)+ix\,{\hat p_{-1}}(y)-{\hat p_{-2}}(y-x)\\
 & = & -\pi\abs{y}+\pi x\,{\rm sgn}(y)+\pi\abs{y-x}.
\end{eqnarray*}
See, for example, \cite[8.3.7, 8.3.19]{friedlanderjoshi} for the Fourier transforms
of $p_{-1}$ and $p_{-2}$.

If $x\geq 0$ then
$$
\frac{1}{2\pi}\Omega_{\fhat}(-x)=\frac{1}{2}\intinf(-\abs{y}+x\,{\rm sgn}(y)+\abs{y-x})f(y)\,dy
=\int_0^x(x-y)f(y)\,dy.
$$
For each $x\geq 0$ then
$$
\frac{1}{2\pi}\frac{d}{dx}\Omega_{\fhat}(-x)=\int_0^xf(y)\,dy.
$$
And, at points of continuity of $f$, \eqref{dx^2} holds.  Similarly for $x\leq 0$.

For the fundamental theorem of calculus at points of approximate continuity, see,
for example, \cite[Theorem~14.8]{gordon}.
\end{proof}

\section{Examples}\label{sectionexamples}

If $f\in L^1(\R)$ then $\fhat$ is continuous with limit $0$ at $\pm\infty$.
And, if $f\in L^p(\R)$ for some $1<p\leq 2$ then $\fhat\in L^q(\R)$ where
$q$ is the conjugate exponent to $p$.  If $f\in L^p(\R)$ for some $2<p\leq \infty$ 
then $\fhat$ exists as a tempered distribution but need not be a function
\cite[4.13, p.~34]{steinweiss}.
When $1<p<\infty$ then $\fhat$ is the distributional derivative of a H\"older 
continuous function \cite{talvilaLpFourier}.  

The following examples show some
of the different types of behaviour that Fourier transforms can have for $f\in\alexc$ when
$f$ is not in an $L^p(\R)$ space.  In the final example, the exchange theorem
is used to compute an integral that does not appear in tables.

The theorems in this paper apply to $\alexc$.  This space includes functions
that are not in $L^1_{loc}(\R)$, such as $f(x)=x^{-1}\cos(x^{-2})$.
It includes distributions that have no pointwise values.  For example,
there is a function $F\in\balexc$ of Weierstrass type, i.e., the pointwise
derivative of $F$ exists nowhere.  The distributional derivative is $F'\in\alexc$
and $F'$ has a Fourier transform according to  Definition~\ref{defnfhatalexc}.
As well, there is a function $G\in\balexc$, akin to the Cantor--Lebesgue 
function, such that $G\not=0$ and the pointwise derivative $G'(x)=0$ for almost all $x\in\R$.
Then the Lebesgue integral $\intinf e^{-ist}G'(t)\,dt=0$ for all $s$ but
$G'\in\alexc$ and the Fourier transform of $G'$ is well defined and is not $0$.

Various special cases are considered in the following examples.

\begin{example}[$\fhat$ not locally integrable]\label{examplenotlocallyintegrable}
If $f\in\hk$ then part (d) of Theorem~\ref{theoremOmegaf} shows $\Omega_f$ need
not be Lipschitz continuous and thus need not be in $C^1(\R)$.  Then $\fhat$ need
not be the distributional derivative of a continuous function so that
$\int_a^b\fhat$ can fail to exist as an integral in $\alexc$ (and hence also
fail to exist as a Lebesgue, improper Riemann or Henstock--Kurzweil integral).
Similarly, $\fhat$ need not be locally integrable in the sense of regulated
primitive integrals \cite{talvilaregulated}.
\end{example}

Use the notation $e_x(y)=e^{xy}$.

\begin{example}[Pointwise divergence everywhere]\label{examplekolmogorov}
Kolmogorov \cite{kolmogorov1926} gave an example of a periodic $L^1$ function
such that the symmetric partial sums of the Fourier series diverge everywhere.
According to remarks on transference principles \cite[p.~481]{grafakosmodern}, this
implies the existence of $h\in L^1(\R)$ such that 
$\lim_{R\to\infty}\int_{-R}^Re^{ixs}\hhat(s)\,ds$
diverges for each $x\in\R$.
(See also Matt Rosenzweig's and David C. Ullrich's answers to 
a question posed on StackExchange-Mathematics \cite{rosenzweigullrich}.)
  Note that $\hhat\in\balexc$.  Define $f\in\alexc$
by $f=\hhat'$.   Integration by parts shows 
$\lim_{R\to\infty}\int_{-R}^Re_{ix}f$ diverges for each $x\in\R\setminus\{0\}$.
The Hake Theorem \cite[Theorem~25]{talviladenjoy} now shows
$\intinf e_{ix}f$ diverges for each $x\in\R\setminus\{0\}$.
Since $\fhat$ exists in $\Dc$,
Definition~\ref{defnfhatalexc} properly extends the definition of 
the Fourier transform as defined by the integral $\intinf e^{-ist}f(t)\,dt$.
\end{example}

\begin{example}[Compact support]\label{exampleccompactsupport}
If $f\in\alexc$ with primitive $F\in\balexc$ then $f$ is of compact support
if there are $-\infty<\alpha<\beta<\infty$ such that $F$ is $0$ on 
$[-\infty,\alpha]$ and constant on $[\beta,\infty]$.  Then
$\Omega_f(s)=\int_\alpha^\beta v_sf=F(\beta)v_s(\beta)-
\int_\alpha^\beta  v_s'(t)F(t)\,dt$.  It follows by dominated convergence that
$\Omega_f$ is analytic and then so is $\fhat$.

In this case, 
\begin{eqnarray}
\fhat(s) & = & \Omega_f''(s)=\int_\alpha^\beta  
\frac{\partial^2 v_s}{\partial s^2}f=\int_\alpha^\beta e_{is}f\label{compact1}\\
 & = & \frac{\partial^2}{\partial s^2}\left[
F(\beta)v_s(\beta)-
\int_\alpha^\beta  v_s'(t)F(t)\,dt\right]
 =  F(\infty)e^{-is\beta}+is\int_\alpha^\beta e^{-ist}F(t)\,dt.\label{compact2}
\end{eqnarray}
For justification of differentiation inside the integral in 
\eqref{compact1}, see \cite[Theorem~22]{talviladenjoy}.  This shows the
Fourier transform is given by the usual integration formula.  Dominated convergence
allows differentiation inside the integral in \eqref{compact2} and this
represents the Fourier transform as the integration of a continuous function on
a compact interval.  The integral $\fhat=\int_\alpha^\beta e_{is}f$ exists in
$\alexc$ since $e_{is}$ is of bounded variation on $[\alpha,\beta]$.

There is the estimate
$\abs{\fhat(s)}\leq\abs{F(\infty)}+(\beta-\alpha)\abs{s}\norm{F}_\infty$.
The Riemann--Lebesgue lemma shows $\fhat(s)=o(s)$ as $\abs{s}\to\infty$.
Titchmarsh \cite{titchmarshgrowth} has shown this estimate is sharp.
This then provides a more precise estimate than is possible with arbitrary
distributions of compact support (Paley--Wiener--Schwartz Theorem 
\cite[Theorem~10.2.1]{friedlanderjoshi}).
\end{example}

\begin{example}
Let $f(x)=\sin(x)/\abs{x}$.  Then $f\in\hk$ and
$\fhat(s)=i\log(\abs{s-1}/\abs{s+1})$ \cite[2.6(1)]{erdelyi}.
Hence, Fourier transforms in $\alexc$ can
be unbounded on compact intervals.
\end{example}

Even when $\fhat$ is continuous it need not be in $\hk$ or any $L^p(\R)$ space.
\begin{example}\label{exampleAMM}
Let $f(x)=x^2e^{ix^4}$ for $x\geq 0$ and $f(x)=0$ for $x\leq 0$.  Integration
by parts establishes that $f\in\hk$ but 
for no $1\leq p\leq\infty$ is
$f\in L^p(\R)$.  Integration by parts also shows $\fhat$ is continuous.  
For $s>0$ we can write
$$
\fhat(s)=\int_0^\infty t^2e^{i(t^4-st)}\,dt=s\int_0^\infty t^2e^{is^{4/3}(t^4-t)}\,dt.
$$
The function $\phi(x)=x^4-x$ has its minimum at $x=4^{-1/3}$.  By the method of
stationary phase there are non-zero constants $A$ and $B$ such that the
asymptotic behaviour is
$$
\fhat(s)\sim As^{1/3}e^{iBs^{4/3}} \quad\text{ as }
s\to\infty.
$$
It follows that $\fhat$ is not in $\hk$ or any $L^p(\R)$ space, for $1\leq p\leq\infty$.
Notice also the failure of the Riemann--Lebesgue lemma since $\fhat$ does not
have a limit at infinity.  

See \cite{talvilamaarapid} for an example of an $\hk$
Fourier transform having arbitrarily large pointwise growth on a sequence.
\end{example}

\begin{example}\label{exampleintegralcomputation}
The exchange theorem can be used to evaluate definite integrals.

Let $f(t)=\abs{t}^{\nu-1}\sin(x\abs{t}+\nu\pi/2)$ with $x>0$ and $0<Re(\nu)<1$.
Then $f\in\hk$ and for each
$1\leq p\leq\infty$ we have $f\not\in L^p(\R)$.  The Fourier transform is
\cite[1.6(2), 1.6(17)]{erdelyi}
$$
\fhat(s)=\Gamma(\nu)\sin(\nu\pi)\left\{\begin{array}{cl}
(x+s)^{-\nu}+(x-s)^{-\nu}, & 0<s<x\\
(x+s)^{-\nu}, & s>x,
\end{array}
\right.
$$
with $\fhat$ extended as an even function to $(-\infty,0)$.

Let $g(t)=\log(1+y^2/t^2)$ for $y>0$.  Then $\ghat(s)=2\pi(1-e^{-y\abs{s}})/\abs{s}$
\cite[1.5(10)]{erdelyi}.  And, $g$ satisfies the conditions of Theorem~\ref{theoremexchange}.

To compute $\intinf f(t)\ghat(t)\,dt$ note that
the integral $\int_0^\infty t^{\nu-2}\sin(xt+\nu\pi/2)\,dt$ exists for $1<Re(\nu)<2$ and
the integral
$\int_0^\infty t^{\nu-2}\sin(xt+\nu\pi/2)e^{-yt}\,dt$ exists for
$Re(\nu)>1$ but their difference exists for $0<Re(\nu)<2$ and this lets
us compute
\begin{align*}
&\intinf f(t)\ghat(t)\,dt=4\pi\int_0^\infty t^{\nu-2}\sin(xt+\nu\pi/2)(1-e^{-yt})\,dt\\
&\quad=\frac{4\pi^2\cos(\nu\pi)x^{1-\nu}}
{\Gamma(2-\nu)\sin(\nu\pi)}
-\frac{4\pi^2(x^2+y^2)^{(1-\nu)/2}\sin[(1-\nu)\arctan(x/y)-\nu\pi/2]}
{\Gamma(2-\nu)\sin(\nu\pi)}.
\end{align*}
The gamma functions in \cite[1.3(1), 1.4(7), 2.3(1), 2.4(7)]{erdelyi} have been
written so that these formulas can be analytically continued to the region $0<Re(\nu)<2$.

After a linear change of variables,
$\intinf\fhat(s)g(s)\,ds=2\Gamma(\nu)\sin(\nu\pi)\int_0^\infty s^{-\nu}
\log(1+y^2/(s-x)^2)\,ds$.  The
exchange theorem gives
\begin{align}
&\int_0^\infty s^{-\nu}\log(1+y^2/(s-x)^2)\,ds\label{exchangeparts}\\
&\quad=\frac{2\pi^2\cos(\nu\pi)x^{1-\nu}}
{\Gamma(\nu)\Gamma(2-\nu)\sin^2(\nu\pi)}
-\frac{2\pi^2(x^2+y^2)^{(1-\nu)/2}\sin[(1-\nu)\arctan(x/y)-\nu\pi/2]}
{\Gamma(\nu)\Gamma(2-\nu)\sin^2(\nu\pi)}.
\end{align}

This integral does not appear in \cite{apelblat}, \cite{erdelyi}, \cite{gradshteyn},
\cite{grobnerhofreiter}
or \cite{prudnikov}.

If $x<0$ then integration by parts can be applied in \eqref{exchangeparts} and the 
resulting integrals
appear as \cite[6.2(3), 6.2(12)]{erdelyi}.  The formulas here can be analytically
continued to the region  $0<Re(\nu)<1$ using gamma function identities.
\end{example}

\section{Convolution}\label{sectionconvolution}
If $f_1,f_2\in L^1(\R)$ then $\widehat{f_1\ast f_2}=\widehat{f_1}\widehat{f_2}$.
Versions of this formula are given in $\alexc$.
Recall Definition~\ref{defnpn}, $p_n(x)=x^n$.

\begin{theorem}\label{theoremconvolution1}
Let $f\in\alexc$.  Let $g_1$ and $g_2$ be functions such that
(A) $g_1$ is absolutely continuous,
(B) $p_2g_1\in L^1(\R)$,
(C) $p_1g_1\in\bv$,
(D) $p_2g_1'\in L^1(\R)$,
(E) $p_1g_1'\in\bv$,
(F) $p_2g_1'\in\bv$,
(G) $g_2$ is absolutely continuous,
(H) $p_1g_2\in L^1(\R)$,
(I) $p_2g_2\in\bv$,
(J) $p_2g_2'\in\bv$,
(K) $g_2'\in\bv$.
Then $\intinf\widehat{f\ast g_1}g_2=\intinf\fhat\gonehat g_2$.
\end{theorem}
A simplified set of assumptions appears in Corollary~\ref{corollaryconvolution}.
\begin{proof}
By (A) and (B), $g_1\in L^1$.
Since $f\in\alexc$ the convolution is then well defined and $f\ast g_1\in\alexc$ 
\cite[Definition~3.1]{talvilaconvolution}.  Notice that $g_2$ satisfies the
conditions of Corollary~\ref{corollaryexchange}.  Then 
$\intinf\widehat{f\ast g_1}g_2=\intinf f\ast g_1 \gtwohat$.

Now show
$\gonehat g_2$ satisfies the conditions of 
Corollary~\ref{corollaryexchange}.  We have $\norm{\gonehat g_2}_1\leq\norm{g_1}_1\norm{g_2}_1$.
From (A) and (B), $\gonehat\in C^1(\R)$ so $\gonehat g_2$ is absolutely continuous on $\R$.
From (G) and (I), $p_1g_2$ is of bounded variation, and
${\rm var}(p_1\gonehat g_2)\leq {\rm var}(\gonehat)\norm{p_1g_2}_\infty+\norm{\gonehat}_\infty
{\rm var}(p_1g_2)$.  Hence,
if $\gonehat\in\bv$ then $p_1\gonehat g_2\in\bv$.  This follows from
(A)-(E) and \cite[Proposition~7.2]{talvilaLpFourier}(d).  And, 
$$
{\rm var}(p_2(\gonehat g_2)') \leq {\rm var}(\gonehat')\norm{p_2g_2}_\infty + \norm{\gonehat'}_\infty
{\rm var}(p_2g_2)+{\rm var}(\gonehat)\norm{p_2g_2'}_\infty +\norm{\gonehat}_\infty{\rm var}(p_2g_2').
$$
It remains to show $\gonehat'\in\bv$.  This follows, as in \cite[Proposition~7.2]{talvilaLpFourier}(d),
by (A)-(D) and (F).
Hence,
$
\intinf \fhat \gonehat g_2=\intinf f \widehat{\gonehat g_2}
$.

A calculation
with the Fubini--Tonelli theorem shows
that, since $g_1, g_2\in L^1(\R)$, then $\widehat{\gonehat g_2}=\widetilde{g_1}\ast \gtwohat$,
where $\tilde{g}(x)=g(-x)$.
Employing the Fubini theorem for $\alexc$, Theorem~\ref{theoremfubini},
$$
\intinf f\widehat{\gonehat g_2}  =  \intinf f(\cdot)\intinf g_1(u-\cdot)
\gtwohat(u)\,du=\intinf f\ast g_1\gtwohat.
$$
This proposition
requires $g_1\in\bv$ ((A), (C)) and $\gtwohat\in L^1(\R)$.  The latter follows from
(G)-(H), (J) and \cite[Proposition~7.2]{talvilaLpFourier}(c).
\end{proof}

\begin{corollary}\label{corollaryconvolution}
Let $f\in\alexc$.  Let $g_1, g_2\in C^2(\R)$ such that 
$\int_{\abs{t}>1}\abs{t^2g_k^{(n)}(t)}dt<\infty$
for $n=0, 1, 2$ and $k=1, 2$.  Then 
$\intinf\widehat{f\ast g_1}g_2=\intinf\fhat\gonehat g_2$.
\end{corollary}

\begin{theorem}
Let $f\in\alexc$.  Let $g_1$ and $g_2$ be absolutely continuous functions such that
$p_1g_1, p_1g_2\in L^1(\R)$ and $p_2g_1', p_2g_2'\in L^1(\R)$.
Then $\intinf \fhat g_1\ast g_2=\intinf f\gonehat\gtwohat$.
\end{theorem}

\begin{proof}
The hypotheses show that $g_1, g_2\in L^1(\R)$; $g_1', g_2'\in L^1(\R)$ and
 $p_1g_1', p_1g_2'\in L^1(\R)$.
Let $g=g_1\ast g_2$.  Then $g\in L^1(\R)$ and is
absolutely continuous.  To use Corollary~\ref{corollaryexchange} show that
$p_1g$ and $p_2g'$ are of bounded variation.

We have
\begin{eqnarray*}
{\rm var}(p_1g) & \leq & \intinf \abs{s}\intinf\abs{g_1(s-t)g_2'(t)}\,dt\,ds
+  \intinf\intinf\abs{g_1(s-t)g_2(t)}\,dt\,ds\\
 & \leq & \intinf(\abs{s-t}+\abs{t})\intinf\abs{g_1(s-t)g_2'(t)}\,dt\,ds
+\norm{g_1}_1\norm{g_2}_1\\
 & = & \norm{p_1g_1}_1\norm{g_2'}_1+\norm{g_1}_1\norm{p_1g_2'}_1+\norm{g_1}_1\norm{g_2}_1.
\end{eqnarray*}

And,
\begin{align*}
&{\rm var}(p_2g')  \leq  2\intinf \abs{s}\intinf\abs{g_1(s-t)g_2'(t)}\,dt\,ds
+\intinf s^2\intinf\abs{g_1'(s-t)g_2'(t)}\,dt\,ds\\
 & \leq  2\intinf(\abs{s-t}+\abs{t})\intinf\abs{g_1(s-t)g_2'(t)}\,dt\,ds\\
  & \quad+\intinf((s-t)^2+2\abs{s-t}\abs{t}+t^2)\intinf\abs{g_1'(s-t)g_2'(t)}\,dt\,ds\\
 & \leq  2\norm{p_1g_1}_1\norm{g_2'}_1+2\norm{g_1}_1\norm{p_1g_2'}_1
+\norm{p_2g_1'}_1\norm{g_2'}_1
+2\norm{p_1g_1'}_1\norm{p_1g_2'}_1
  +\norm{g_1'}_1\norm{p_2g_2'}_1.
\end{align*}
\end{proof}

\end{document}